\documentclass[a4paper]{article}

\usepackage{amsthm,amssymb,amsmath,enumerate,graphics}

\newtheorem{theorem}{Theorem}

\newtheorem{conjecture}{Conjecture}

\newtheorem{problem}[conjecture]{Problem}

\theoremstyle{remark}
   \newtheorem{example}{Example}

\newcommand{\comment}[1]{}

\newcommand{\R}{\mathbb R}

\newcommand{\Tr}[1]{Theorem~\ref{#1}}
\newcommand{\Sr}[1]{Section~\ref{#1}}

\newcommand{\pmg}{perfect matching}

\renewcommand{\mod}{mod\ }

\title{Perfect matchings in $r$-partite $r$-graphs}

\author{Ron Aharoni\\
Technion Institute of Technology 
\and Agelos Georgakopoulos\\
Mathematisches Seminar, Universit\"at Hamburg
 \and Philipp Spr\"ussel\\
Mathematisches Seminar, Universit\"at Hamburg}

\date{}
\begin{document}
\maketitle

\begin{abstract}
Let $H$ be an $r$-partite $r$-graph, all of whose sides have the
same size $n$. Suppose that there exist two sides of $H$, each
satisfying  the following condition: the degree of each legal
$(r-1)$-tuple contained in the complement of this side is strictly
larger than $\frac{n}{2}$. We prove that under this condition $H$
must have a perfect matching. This answers a
question of K\"uhn and Osthus.
\end{abstract}

\section{Introduction}
Matchings in hypergraphs are notoriously evasive. There is an
abundance of conjectures in the subject, and no well developed
theory similar to matching theory in graphs. In this paper we prove the sufficiency of a
certain condition, for the existence of a perfect matching in an
$r$-partite $r$-graph. This is a generalization of the well
known result that if in an $n \times n$ bipartite graph the degree
of every vertex is at least $\frac{n}{2}$ then the graph has a
perfect matching.

We will be using the terminology of Diestel \cite{diestelBook05}. An $r$-uniform hypergraph $H$
(also referred to as an $r$-{\em graph}) is said to be $r$-{\em partite} if its
vertex set $V(H)$ can be partitioned into sets $V_1, V_2, \ldots,V_r$,
called the ``sides" of $H$, so that every edge in the edge set $E(H)$ of
$H$ consists of a choice of precisely one vertex from each side.
This means that $E(H) \subseteq V_1 \times V_2\times \ldots \times
V_r$, in particular that the edges of $H$ can be
considered as ordered $r$-tuples.

The {\em degree} $d(f)$
in $H$ of a subset $f$ of $V$ is the number of edges of $H$
containing $f$. An $r$-partite hypergraph is said to be $n$-{\em
balanced} if $|V_i|=n$ for every $1 \le i \le r$. A set of
vertices is called {\em legal} if it meets each side in at most
one vertex.

In \cite{kuhnosthus} K\"uhn and Osthus proved the following:
\begin{theorem}
If in an $n$-balanced $r$-partite $r$-graph $H$ every legal $(r-1)$-tuple has degree at least $n/2 + \sqrt{2n\log n}$ and $n \geq 1000$ then $H$ has a
perfect matching.
\end{theorem}

 The following example of K\"uhn and Osthus shows
that demanding that every legal $(r-1)$-tuple has degree at least $n/2$ does not suffice for the existence of a perfect matching:

\begin{example}\label{sharpnessofnover2}
Suppose that $r$ is odd, and that $n$ is even but not divisible by
$4$. For every $i\le r$  choose a subset $A_i$ of $V_i$ of size
$\frac{n}{2}$. Let $H$
be the hypergraph containing precisely those legal $r$-tuples that contain an even number of vertices in $\bigcup_{i \le r}A_i$. Then $d(f)=\frac{n}{2}$ for every legal $(r-1)$-tuple $f$. However, every matching contains an even number of vertices of $\bigcup_{i \le r}A_i$, and since $|\bigcup_{i \le r}A_i|$ is odd there can be no perfect matching in $H$.

For all other values of $r$ and $n$ choose $A_i$ as above such
that $||A_i|-\frac{n}{2}|\le 1$ and $\sum|A_i|$ is odd. This
yields an $r$-partite $r$-graph such that $d(f)\ge \frac{n}{2}-1$
for every legal $(r-1)$-tuple $e$, that has no perfect matching.
\end{example}

K\"uhn and Osthus \cite{kuhnosthus} posed the question whether a minimal degree greater than $\frac{n}{2}$ forces a perfect matching. It is the aim of this paper to prove this assertion, in a somewhat
stronger form:

\begin{theorem} \label{main}
Let $H$ be an $n$-balanced $r$-partite $r$-graph with partition classes $V_1, \ldots , V_r$. If for every legal $(r-1)$-tuple $f$
contained in $V \setminus V_1$ we have $d(f)>\frac{n}{2}$ and for every legal $(r-1)$-tuple $g$
contained in $V \setminus V_r$ we have $d(g)\geq\frac{n}{2}$  then $H$ has a \pmg.
\end{theorem}

Example~\ref{sharpnessofnover2} suggests that, possibly, if $r$ is even or $n \neq 2 (\mod 4)$, then the degree condition in \Tr{main} can be relaxed to that of every legal $(r-1)$-tuple having degree at least $\frac{n}{2}$. We do not know whether this is true. In \Sr{problems} we propose some further problems.

In this paper we restricted our attention to $r$-partite hypergraphs. Forcing perfect matchings by large minimum degree of $(r-1)$-tuples in $r$-uniform graphs in general has been an active field lately, see \cite{rodlRuciSzeme} for example.

\section{Proof of \Tr{main}}

In this section we prove \Tr{main}.

\begin{proof}
As noted in \cite{kuhnosthus}, it suffices to prove the theorem
for $r=3$. To see this, let $r>3$ and choose a perfect matching $F=g_1, g_2, \ldots, g_n$ in the complete
$(r-2)$-partite $(r-2)$-graph with vertex partition $V_2 , V_3, \ldots, V_{r-1}$. Let $H'$ be the $3$-partite $3$-graph with vertex partition $V_1 , F , V_r$ where $(x, g_i ,y)$ is an edge of $H'$ if and only if $\{x\}\cup g_i \cup \{y\}$ is an edge of $H$ (where $x \in V_1,~y \in V_r$) . Clearly, $H'$ satisfies the conditions of
the theorem, with $r=3$. Assuming that the theorem is
valid in this case, $H'$ has a perfect matching, and``de-contracting" each $g_i$ results in a perfect matching of $H$.

Thus we may assume that $r=3$. Suppose that the theorem
fails.  By considering a counterexample with maximal set of edges
we may assume that $H$ has a matching $M$ that matches all but one
vertex from each class; let $x_1 \in V_1, x_2 \in V_2, x_3 \in
V_3$ be the unmatched vertices.

Let $U$ be the set of pairs $(u,v)$ where $u\in V_2,
v \in V_3$ and there is an edge of $M$ containing both $u$ and $v$. Since each pair in $U$
has more than $\frac{n}{2}$ neighbors in $V_1$, there exists a
vertex $w\in V_1$ that is a neighbor of at least $\frac{n}{2}$
pairs in $U$. We consider three cases, in all of which we will be
able to construct a perfect matching of $H$.

The first case is when  $w=x_1$. Since the pair $(x_2,x_3)$ has more than
$\frac{n}2$ neighbors in $V_1$, there is an edge
$e=(u_1,u_2,u_3)\in M$ such that $(x_1,u_2,u_3) \in H$ and $(u_1,x_2,x_3) \in H$. Then  $M - e +
(x_1,u_2,u_3)+(u_1,x_2,x_3)$ (standing for  $M \setminus \{e\}
\cup \{(x_1,u_2,u_3), (u_1,x_2,x_3)\}$) is a perfect matching of
$H$.

The next case is when $w$ lies on an edge $f=(w,u_2,u_3)$ of $M$ such
that $(x_1,x_2,u_3) \in E(H)$. Since the pair $(u_2,x_3)$ has
more than $\frac{n}2$ neighbors, there is an edge $g=(v_1,v_2,v_3)
\in M$ such that $v_1$ is a neighbor of the pair $(u_2,x_3)$ and
the element $(v_2,v_3)$ of $U$ is in an edge with $w$. If $v_1=w$ (in
which case $f=g$) then $M - g + (x_1,x_2,v_3)+(v_1,v_2,x_3)$ is a
perfect matching of $H$, and if $v_1\neq w$ then $M - f - g
+(x_1,x_2,u_3)+ (v_1,u_2,x_3)+(w,v_2,v_3)$ is a perfect matching.

Finally, consider the case when $w$ lies in an edge
$f=(w,u_2,u_3)$ of $M$ such that $(x_1,x_2,u_3) \not\in E(H)$. Since
$d((u_2, u_3))>n/2$ and $d((x_1,x_2))\geq n/2$ there is an edge $g=(v_1,v_2,v_3) \in M$
such that $(v_1,u_2, u_3)\in E(H)$ and $(x_1,x_2,v_3)\in E(H)$. Let $M'$
be  the matching $M - f - g + (v_1,u_2,u_3)+(x_1,x_2,v_3)$. The
only vertices not matched by $M'$ are $v_2,x_3$ and $w$. Now we
can repeat the argument of the first case with $w$ playing the
role of $x_1$. But in this case we have to be more careful: as $w$
was a neighbor of at least $\frac{n}2$ pairs in $U$, and the only
element of $U$ that is not in an edge of $M'$ is $(v_2,v_3)$, there
are still at least $\frac{n}2 - 1$ elements of $U$ neighboring $w$
that are each in an edge of $M'$. On the other hand, if
$(w,v_2,x_3)\in E(H)$ we are done. Hence we can assume that the pair
$(v_2,x_3)$ has at least $\frac{n+1}2$ neighbors in $V_1 - w$. But
$\frac{n}2-1+\frac{n+1}2>n-1 = |M'|$, thus there is an edge $e$ of $M'$
containing a pair neighboring $w$ and a neighbor of $(v_2,x_3)$.
Removing $e$ from $M'$ and adding the two corresponding edges
yields a \pmg\ of $H$.

\end{proof}

\section{Open problems} \label{problems}

The condition in \Tr{main}, although sharp for infinitely many values of $n$ and $r$, is very strong.
It is likely that it can be weakened, in more than one way. We offer some conjectures as possible weakenings of the condition. Let $H$ be an $n$-balanced $r$-partite $r$-graph fixed throughout this section. For a subset $I$ of $[r]:=\{1,2,\ldots,r\}$ an $I$-{\em tuple} is
an element of $\times_{i \in I}V_i$. Let $I^c:= [r]\setminus I$.

\begin{conjecture}\label{ij}
Let $I$ be a subset of $[r]$. If $d(f)>\frac{n^{r-|I|}}{2}$ for
every $I$-tuple $f$ and  $d(g) \ge \frac{n^{|I|}}{2}$ for every
$I^c$-tuple $g$ (i.e.\ each $I$-tuple has degree
larger
 than half its degree in the complete $r$-partite hypergraph and each
$I^c$-tuple has degree at least  half its degree in the complete
$r$-partite hypergraph) then $H$ has a \pmg.
\end{conjecture}

A stronger version of Conjecture \ref{ij} is that it suffices to
assume that for every legal $r$-tuple $z$ not belonging to $E(H)$
there holds:
$$\frac{d(z\cap I)}{n^{r-|I|}}+\frac{d(z\cap I^c)}{n^{|I|}}>1.$$

We shall prove a fractional version of this conjecture. A {\em fractional matching} of $H$ is a function $h: E(H) \to \R^+$ such that for every vertex $x$  in $H$ there holds $\sum \{h(e) \mid x\in e\} \leq 1$. We say that $h$ is {\em perfect} if $\sum \{h(e) \mid x\in e\}= 1$ for every vertex $x$.

\begin{theorem}
Let $I$ be a subset of $[r]$. If $\frac{d(z \cap I)}{n^{r-|I|}}+ \frac{d(z \cap I^c)}{n^{|I|}} \ge 1$ for every legal $r$-tuple $z$ not belonging to $E(H)$ then there exists a perfect fractional matching.
\end{theorem}

\begin{proof}
For a real valued function $f$ and a set $S$ contained in its domain, we write $f[S]$ for $\sum\{f(s) \mid s \in S\}$. A {\em fractional cover} is a function $g: V(H) \to \R_{\ge 0}$ such that $ g[e] \ge 1$ for every $e \in E(H)$.

We have to show that $\nu^*(H)=n$ where $\nu^*(H)$, the {\em fractional matching number of $H$}, is the maximum value of $h[E(H)]$ over all fractional matchings $h$ of $H$. By linear programming duality (see \cite{schrijverBook} for an introduction to the subject), this is equivalent to showing that $\tau^*(H)=n$, namely that $g[V] \ge n$ holds for every fractional cover ($\tau^*(H)$ is the minimum value of $g[V]$ over all fractional covers $g$ of $H$).

So let $g$ be a fractional cover. For every $j \in [r]$ let $\alpha(j)$ be the minimal value of $g$ on $V_j$, and let $v_j$ be a vertex of $V_j$ with $g(v_j)=\alpha(j)$. Also let $\beta=\alpha[I]$ and $\gamma=\alpha[I^c]$.

Consider the $r$-tuple $z=(v_j)_{j \in [r]}$. By the minimality of the $\alpha(j)$'s, we have $g[V] \ge n g[z]$. Hence we may assume that $g[z]=\beta+\gamma < 1$. In particular, we have $z \notin E(H)$.

Write $\frac{d(z \cap I)}{n^{r-|I|}}=\theta$ and $\frac{d(z \cap I^c)}{n^{|I|}}=\zeta$. Call an $I$-tuple $y$ \emph{good} if $y\cup (z\cap I^c) \in E(H)$. Consider the complete $|I|$-partite graph on $\bigcup_{j\in I}V_j$. It is a well known fact (easily proved by induction) that its edge set can be partitioned into $n^{|I|-1}$ perfect matchings. Since there are $\zeta n^{|I|}$ good $I$-tuples, one of those perfect matchings contains at least $\zeta n$ good $I$-tuples; we thus have a set $Y$ of at least $\zeta n$ disjoint good $I$-tuples. For each $j\in I$, denote by $A_j$ the set of vertices in $V_j$ that are contained in an $I$-tuple in $Y$. Since $g[y] \ge 1-\gamma$ for each good $I$-tuple $y$, we have $g[\bigcup_{j\in I}A_j]\ge |Y|(1-\gamma)$. This yields $g[\bigcup_{j\in I}V_j] = g[\bigcup_{j\in I}A_j]+g[\bigcup_{j\in I}(V_j\setminus A_j)] \ge |Y|(1-\gamma)+(n-|Y|)\beta$. Since $\beta < 1-\gamma$ and $|Y|\ge\zeta n$, we obtain $g[\bigcup_{j\in I}V_j] \ge
n\beta + |Y|(1 - \gamma - \beta) \geq n\beta + n \zeta (1-\gamma -\beta) =
\zeta n(1-\gamma)+(1-\zeta)n\beta$. Similarly, we have $g[\bigcup_{j\in I^c}V_j] \ge \theta n(1-\beta)+(1-\theta)n\gamma$ and thus
\begin{align*}
  g[V] &\geq \zeta n(1-\gamma)+(1-\zeta)n\beta+\theta n(1-\beta)+(1-\theta)n\gamma\\
  &= n(\zeta+\theta)+n\beta(1-\zeta-\theta)+n\gamma(1-\zeta-\theta)\\
  &= n\big(1+(\beta+\gamma-1)(1-\zeta-\theta)\big)\\
  &\ge n,
\end{align*}
since $\beta+\gamma-1<0$ and $1-\zeta-\theta \le 0$.
\end{proof}

Let us mention that the problem of forcing perfect fractional  matchings by large minimum degree in $r$-uniform hypergraphs that are not necessarily $r$-partite has been studied in \cite{RRSEuroJ}.

Next we ask what condition on the degrees of vertices, rather than
$I$-tuples, suffices for the existence of a perfect matching in an
$n$-balanced $r$-partite hypergraph.

\begin{problem} \label{prob}
Is it true that if $d(x)\geq (1 - 1/e)n^{r-1}$ for every vertex $x$ of $H$ then there is a \pmg?
\end{problem}

Taking a subset $X_i$ of
$V_i$ of size a bit less than $\frac{n}{r}$ for each $i \in [r]$, and
letting $H$ be the hypergraph consisting of all edges meeting
$\bigcup_{i \in [r]}X_i$, shows that if
the assertion of Problem~\ref{prob} is true then it is asymptotically tight (as $r$ goes to infinity).

Some of the most intriguing conjectures on $3$-partite hypergraphs
were originally formulated in terms of Latin squares. Here is one
of the best known of those, the Brualdi-Ryser conjecture
(\cite{brualdi, ryser}):

\begin{conjecture}
Let $H$ be an $n$-balanced  $3$-partite hypergraph in which every legal
$2$-tuple participates in precisely one edge. If $n$ is odd then there exists a
perfect matching and if $n$ is even there is a matching of size $n-1$.
\end{conjecture}

As Stein pointed out in \cite{steinLatin}, the condition of the Brualdi-Ryser conjecture is probably way too strong, and the conclusion is probably valid
assuming much less than that. Here is a rather bold conjecture of this type:

\begin{conjecture}
Let $H$ be an $n$-balanced $r$-partite $r$-graph, and let $I$ be a subset of $[r]$. If $d(e)=d(f)$ for every two $I$-tuples $e,f$ and $d(g)=d(z)$ for every two $I^c$-tuples $g,z$ then there is a \pmg\ unless $r$ is odd and $n$ even.
\end{conjecture}

Let us mention a result in this direction, in which the assumptions are again probably way too strong:
\begin{theorem}[\cite{aharoni_berger_ziv}]
Let $H$ be a $3$-partite hypergraph, with sides $V_i, i=1,2,3$, where $|V_1|=n$ and $|V_2|\ge 2n-1$. Suppose, furthermore, that the degree of every pair in $(V_1\times V_2)$  is $1$ and the degree of every pair in $(V_1 \times V_3)$ is at most $1$. Then there exists in $H$ a matching of size $n$.
\end{theorem}

\bibliographystyle{plain}
\bibliography{graphs}

\small
\vskip2mm plus 1fill
\parindent=0pt\obeylines

\bigbreak
Ron Aharoni {\tt <ra@tx.technion.ac.il>}
\smallskip
Department of Mathematics
Technion, Haifa
Israel 32000
\smallskip

Agelos Georgakopoulos {\tt <georgakopoulos@math.uni-hamburg.de>}
Philipp Spr\"ussel {\tt <philipp.spruessel@gmx.de>}
\smallskip
Mathematisches Seminar
Universit\"at Hamburg
Bundesstra\ss e 55
20146 Hamburg
Germany

\end{document}